\documentclass[12pt,a4paper]{article}
\usepackage{amsmath,amssymb,amsthm,graphicx,color,titling}
\usepackage[left=1in,right=1in,top=1in,bottom=1in]{geometry}
\usepackage{cite}
\usepackage[british]{babel}

\newtheorem{theorem}{Theorem}
\newtheorem{lemma}[theorem]{Lemma} 
 
\newtheorem{proposition}[theorem]{Proposition}

\theoremstyle{definition}

\theoremstyle{remark}

\allowdisplaybreaks

\newcommand{\R}{{\mathbb{R}}}

\newcommand\Exp{{\mathbb{E}}}
\renewcommand\Pr{{\mathbb{P}}}

\newcommand\1{{\mathbf 1}}

\newcommand{\ud}{{\mathrm{d}}}
\newcommand\be{{\mathbf{e}}}
\newcommand\bz{{\mathbf{z}}}

\newcommand{\cH}{{\mathcal{H}}}

\newcommand{\hull}{\mathop{{\rm hull}}}
\newcommand{\diam}{\mathop{{\rm diam}}}

\title{On the expected diameter of planar Brownian motion}
\thanksmarkseries{alph}
\author{James McRedmond\thanks{Department of Mathematical Sciences, Durham University, South Road, Durham DH1 3LE, UK.}\hspace{0.4in} Chang Xu\thanks{Department of Mathematics and Statistics, University of Strathclyde, 26 Richmond Street, Glasgow G1 1XH, UK.}} 
\date{\today}

\begin{document}

\maketitle

\begin{abstract}
Known results show that the diameter $d_1$ of the trace of planar Brownian
motion run for unit time satisfies $1.595 \leq \Exp d_1 \leq 2.507$.
This note improves these bounds to $1.601 \leq \Exp d_1 \leq 2.355$.
Simulations suggest that $\Exp d_1 \approx 1.99$.
\end{abstract}



  \bigskip

Let $(b_t, t \in [0,1])$ be standard planar Brownian motion, and consider
the set
$b[0,1] = \{ b_t : t \in [0,1] \}$. The Brownian convex hull $\cH_1 := \hull b [0,1]$ has been well-studied
from L\'evy~\cite[\S52.6, pp.~254--256]{levy} onwards; the expectations of the perimeter length $\ell_1$ and area $a_1$ of $\cH_1$ are given by the exact formulae  $\Exp \ell_1 = \sqrt{8\pi}$ (due to
Letac and T\'akacs \cite{letac,takacs}) and $\Exp a_1 = \pi /2$ (due to El Bachir \cite{bachir}).

Another characteristic is the \emph{diameter} \[d_1 := \diam \cH_1 = \diam b [0,1]=\sup_{x,y\in b[0,1]}\|x-y\|,\]
for which, in contrast, no explicit formula is known. The exact formulae for $\Exp \ell_1$
and $\Exp a_1$ rest on geometric integral formulae of Cauchy; since no such formula
is available for $d_1$, it may not be possible to obtain an explicit formula for $\Exp d_1$.
However, one may get bounds.

By convexity, we have the almost-sure inequalities
$2 \leq \ell_1/ d_1 \leq \pi$,
the extrema being the line segment and shapes of constant width (such as the disc). In other words, 
\[ \frac{\ell_1}{\pi} \leq d_1 \leq \frac{\ell_1}{2} .\]
The formula of Letac and Tak\'acs \cite{letac,takacs} says that $\Exp \ell_1 = \sqrt{ 8 \pi}$, so
we get:
\begin{proposition}
\label{prop1}
$\sqrt{ 8 / \pi } \leq \Exp d_1 \leq \sqrt{2 \pi}$.
\end{proposition}

Note that $\sqrt{8 / \pi} \approx 1.5958$ and $\sqrt{2\pi} \approx 2.5066$.
In this note we improve both of these bounds. 

For the lower bound, we note that $b[0,1]$ is compact and thus, as a corollary of Lemma \ref{lem:diam} below, we have the formula
\begin{equation}
\label{eq:d-formula}
 d_1 = \sup_{0 \leq \theta \leq \pi} r (\theta)  ,
\end{equation}
where $r$ is the parametrized range function given by
\[ r (\theta) = \sup_{0 \leq s \leq 1} \left(b_s \cdot \be_\theta\right)
 - \inf_{0 \leq s \leq 1} \left(b_s \cdot \be_\theta\right), \]
with $\be_\theta$ being the unit vector $(\cos \theta, \sin \theta)$. Feller~\cite{feller} established that
\begin{equation}
\label{eq:feller-bounds}
 \Exp r (\theta) =   \sqrt{ 8 /\pi } \quad \text{ and } \quad \Exp ( r(\theta)^2 ) = 4 \log 2  ,
\end{equation}
and the density of $r(\theta)$ is given explicitly as
\begin{equation}
\label{eq:density}
f (r ) = \frac{8}{\sqrt{2 \pi}} \sum_{k=1}^\infty (-1)^{k-1} k^2 \exp \{ - k^2 r^2 /2 \} , ~ ( r \geq 0).
\end{equation}
Combining~\eqref{eq:d-formula} with~\eqref{eq:feller-bounds} gives immediately
$\Exp d_1 \geq \Exp r (0) = \sqrt{ 8 /\pi }$, which is just the lower bound 
in Proposition~\ref{prop1}. 
For a better result, a consequence of~\eqref{eq:d-formula} is that
 $d_1 \geq \max \{ r(0), r(\pi/2) \}$. Observing that $r(0)$ and $r(\pi/2)$ are independent, we get:

\begin{lemma}
\label{lem1}
$\Exp d_1 \geq \Exp \max \{ X_1, X_2 \}$, where $X_1$ and $X_2$
are independent copies of $X := r(0)$.
\end{lemma}

It seems hard to explicitly compute $\Exp \max \{ X_1, X_2 \}$ in Lemma~\ref{lem1}, because although the density given at~\eqref{eq:density}
is known explicitly, it is not very tractable. Instead we obtain a lower bound.
Since
\[ \max \{ x , y \} = \frac{1}{2} \left( x + y + | x -y | \right) \]
we get
\begin{equation}
\label{eq:max}
\Exp \max \{ X_1, X_2 \} = \Exp X + \frac{1}{2} \Exp | X_1 - X_2 | .
\end{equation}
Thus with Lemma~\ref{lem1}, the lower bound in Proposition~\ref{prop1} is improved
given any non-trivial lower bound for $\Exp | X_1 - X_2 |$. 
Using the fact that for any $c \in \R$, if $m$ is a median of $X$, 
$\Exp | X - c | \geq \Exp | X - m |$,
we see that 
\begin{align*}
\Exp | X_1 - X_2 |   \geq \Exp | X - m | .
\end{align*}
Again, the intractability of the density at~\eqref{eq:density} makes it hard to exploit this.
Instead, we provide the following as a crude lower bound on $\Exp | X_1 - X_2 |$.

\begin{lemma}
\label{lem:ineq1}
For any $a,h >0$,
\begin{equation} \Exp | X_1 - X_2 | \geq 2 h \,\Pr ( X \leq a )\, \Pr ( X \geq a + h) . \nonumber \label{eqn:ineq1}
\end{equation}
\end{lemma}
\begin{proof}
We have
\begin{align*}
\Exp | X_1 - X_2 | 
& \geq \Exp\left[ |X_1-X_2|\1\{X_1\leq a, X_2 \geq a+h\} \right]\\
&{}\qquad{}+\Exp\left[ |X_1-X_2|\1\{X_2\leq a, X_1 \geq a+h\} \right] \\
& \geq h \,\Pr(X_1 \leq a)\,\Pr(X_2\geq a+h)+h \,\Pr(X_2\leq a)\,\Pr(X_1\geq a+h) \\
& = 2h \,\Pr(X\leq a)\,\Pr(X\geq a+h),
\end{align*} 
which proves the statement.
\end{proof}
This lower bound yields the following result.
\begin{proposition}
For $a, h > 0$ define
\[ g ( a, h) :=  h \left( \frac{4}{\pi} \exp \left\{ - \frac{\pi^2}{2a^2 } \right\} - \frac{4}{3\pi}   \exp \left\{ - \frac{9\pi^2}{2a^2 } \right\}  \right)
\left(  1 - \frac{4}{\pi} \exp \left\{ - \frac{\pi^2}{8(a+h)^2 } \right\}  \right) .\]
Then $\Exp d_1 \geq \sqrt{ 8 / \pi} + g ( 1.492, 0.337) \approx 1.6014$.
\end{proposition}
\begin{proof}
Consider 
\[ Z := \sup_{0 \leq s \leq 1} | b_s \cdot \be_0 | . \]
Then it is known (see \cite{jp}) that for $x > 0$, 
\begin{equation}
\label{eq:jp} \frac{4}{\pi} \exp \left\{ - \frac{\pi^2}{8x^2 } \right\} - \frac{4}{3\pi}   \exp \left\{ - \frac{9\pi^2}{8x^2 } \right\} \leq \Pr ( Z < x ) \leq \frac{4}{\pi} \exp \left\{ - \frac{\pi^2}{8x^2 } \right\} .\end{equation}
Moreover, we have
\[ Z \leq X \leq 2 Z .\]
 Since $X \leq 2Z$, we have
\[ \Pr ( X \leq a ) \geq \Pr ( Z \leq a/2) \geq \frac{4}{\pi} \exp \left\{ - \frac{\pi^2}{2a^2 } \right\} - \frac{4}{3\pi}   \exp \left\{ - \frac{9\pi^2}{2a^2 } \right\} ,\]
by the lower bound in~\eqref{eq:jp}. On the other hand,
\[ \Pr ( X \geq a +h ) \geq \Pr ( Z \geq a+h ) \geq 1 - \frac{4}{\pi} \exp \left\{ - \frac{\pi^2}{8(a+h)^2 } \right\} ,\]
by the upper bound in~\eqref{eq:jp}. Combining these two bounds and applying Lemma \ref{lem:ineq1} we get
$\Exp | X_1 - X_2 | \geq 2 g (a , h)$. 
So from~\eqref{eq:max} and the fact that $\Exp X = \sqrt{ 8 / \pi}$ by~\eqref{eq:feller-bounds}
we get $\Exp d_1 \geq \sqrt{ 8 / \pi} + g (a,h)$.
Numerical evaluation using MAPLE suggests that
$(a,h) = (1.492,0.337)$ is close to optimal, and this choice
gives the statement in the proposition.
\end{proof}

We also improve the upper bound in Proposition~\ref{prop1}.

\begin{proposition}
\label{prop3}
$\Exp d_1 \leq  \sqrt{ 8 \log 2 } \approx 2.3548$.
\end{proposition}
\begin{proof}
First, we claim that
\begin{equation}
\label{eq:rectangle-claim}
d_1^2 \leq r(0)^2 + r ( \pi/2)^2 .
\end{equation}
It follows from~\eqref{eq:rectangle-claim} and~\eqref{eq:feller-bounds} that
\[ \Exp ( d_1^2 ) \leq  \Exp ( X_1^2 + X_2^2 ) = 2 \Exp ( X^2 ) = 8 \log 2 . \]
The result now follows by Jensen's inequality.

It remains to prove the claim~\eqref{eq:rectangle-claim}. Note that the diameter is an increasing function, that is, if $A\subseteq B$ then $\diam A \leq \diam B$. Note also, that by the definition of $r(\theta)$, $b[0,1] \subseteq \bz + [0,r(0)]\times[0,r(\pi/2)]=: R_\bz$ for some $\bz\in \R^2$. Since the diameter of the set $R_\bz$ is attained at the diagonal,
\begin{equation*}
\diam R_\bz = \sqrt{r(0)^2 + r(\pi/2)^2},
\end{equation*}
for all $\bz \in \R^2$, and we have $\diam b[0,1] \leq \diam R_\bz$, the result follows.
\end{proof}

We make one further remark about second moments. In the proof of Proposition~\ref{prop3}, we saw that
 $\Exp ( d_1^2 ) \leq 8 \log 2 \approx 5.5452$.
A bound in the other direction can be obtained from the fact that $d_1^2 \geq \ell_1^2 / \pi^2$, and we have
(see~\cite[\S 4.1]{wx2}) that
\[ \Exp ( \ell_1^2 ) = 4 \pi \int_{-\pi/2}^{\pi/2}  \ud \theta \int_0^\infty \ud u  \cos \theta
\frac{ \cosh (u \theta  )}
{ \sinh ( u \pi /2 ) } \tanh \left( \frac{ (2 \theta + \pi) u }{4} \right) \approx 26.1677 ,\]
which gives $\Exp (d_1^2 ) \geq 2.651$.

Finally, for completeness, we state and prove the lemma which was used to obtain equation \eqref{eq:d-formula}.   
\begin{lemma}
\label{lem:diam}
Let $A \subset \R^d$ be a nonempty compact set, and let $r_A(\theta) = \sup_{x \in A} ( x \cdot \be_\theta ) - \inf_{x \in A} (x \cdot \be_\theta)$.
Then
\[ \diam A = \sup_{0 \leq \theta \leq \pi} r_A (\theta ). \]
\end{lemma}
\begin{proof}
Since $A$ is compact, for each $\theta$ there exist $x, y \in A$ such that 
\begin{align*}  r_A (\theta) & = x \cdot \be_\theta - y \cdot \be_\theta \\
& = ( x - y) \cdot \be_\theta  \leq \| x - y \| .\end{align*}
So $\sup_{0 \leq \theta \leq \pi} r_A (\theta) \leq \sup_{x, y \in A} \| x - y \| = \diam A$.

It remains to show that $\sup_{0 \leq \theta \leq \pi} r_A (\theta) \geq \diam A$.
This is clearly true if $A$ consists of a single point, so suppose that $A$ contains at least two points.
Suppose that the diameter of $A$ is achieved by $x, y \in A$ and let $z = y -x$ be such that $\hat z := z/\|z\|= \be_{\theta_0}$ for $\theta_0 \in [0,\pi]$.
Then
\begin{align*} \sup_{0 \leq \theta \leq \pi} r_A (\theta) & \geq r_A (\theta_0 ) \geq y \cdot \be_{\theta_0} - x\cdot \be_{\theta_0} \\
& = z \cdot \hat z = \| z \| = \diam A ,\end{align*}
as required.
\end{proof}

\section*{Acknowledgements}
The authors are grateful to Andrew Wade for his suggestions on this note. The first author is supported by an EPSRC studentship.


\begin{thebibliography}{09}

\bibitem{bachir} M.\ El Bachir, \emph{L'enveloppe convexe du mouvement brownien}, Ph.D. thesis, Universit\'e Toulouse III---Paul Sabatier, 1983.

\bibitem{feller} W.\ Feller, The asymptotic distribution of the range of sums of independent random variables,
\emph{Ann.\ Math.\ Statist.}\ {\bf 22} (1951) 427--432.

\bibitem{jp} N.C.\ Jain and W.E.\ Pruitt,
 The other law of the iterated logarithm, \emph{Ann.\ Probab.}\ {\bf 3} (1975) 1046--1049.

\bibitem{letac} G.\ Letac, Advanced problem 6230, \emph{Amer.\ Math.\ Monthly} {\bf 85} (1978) 686.

\bibitem{levy} P.\ L\'evy, \emph{Processus Stochastiques et Mouvement Brownien},
Gauthier-Villars, Paris, 1948.

\bibitem{takacs} L.\ Tak\'acs, Expected perimeter length, \emph{Amer.\ Math.\ Monthly} {\bf 87} (1980) 142.

\bibitem{wx2} A.R.\ Wade and C.\ Xu, Convex hulls of random walks and their scaling limits, \emph{Stochastic Process.\ Appl.}\ {\bf 125} (2015) 4300--4320.

\end{thebibliography}
\end{document}